\documentclass[11pt,british]{amsart}

\usepackage{babel,eucal,url,amssymb,
enumerate,amscd,
}

\usepackage[pagebackref]{hyperref}

\theoremstyle{plain}
\newtheorem{thm}{Theorem}[section]
\newtheorem{lem}[thm]{Lemma}
\newtheorem{pro}[thm]{Proposition}
\newtheorem{co}[thm]{Corollary}

\theoremstyle{definition}
\newtheorem{defn}[thm]{Definition}

\theoremstyle{remark}
\newtheorem{rem}[thm]{Remark}

\newcommand{\Gtwo}{\ifmmode{{\rm G}_2}\else{${\rm G}_2$}\fi}

\date{\today}

\begin{document}

\title[]
 {The decomposition of almost paracontact metric manifolds in eleven classes revisited}

\author[S. Zamkovoy]{Simeon Zamkovoy}

\address{
University of Sofia "St. Kl. Ohridski"\\
Faculty of Mathematics and Informatics\\
Blvd. James Bourchier 5\\
1164 Sofia, Bulgaria}
\email{zamkovoy@fmi.uni-sofia.bg}

\author[G. Nakova]{Galia Nakova}

\address{
University of Veliko Tarnovo "St. Cyril and St. Methodius" \\ Faculty of Mathematics and Informatics\\   Department of Algebra and Geometry
\\ 2 Teodosii Tarnovski Str. \\ Veliko Tarnovo 5003\\ Bulgaria}
\email{gnakova@gmail.com}

\subjclass{}

\keywords{Almost paracontact metric manifolds, 3-dimensional almost paracontact manifolds, $\alpha$-para-Sasakian manifolds, $\alpha$-para-Kenmotsu manifolds}


\begin{abstract}
This paper is a continuation of our previous work, where eleven basic classes of almost paracontact metric manifolds with respect to the covariant derivative of the structure tensor field were obtained. First we decompose one of the eleven classes into two classes  and  the basic classes of the considered manifolds become twelve. Also, we determine  the classes of $\alpha$-para-Sasakian,  $\alpha$-para-Kenmotsu, normal, paracontact metric, para-Sasakian, K-paracontact and quasi-para-Sasakian manifolds. Moreover, we study 3-dimensional almost paracontact metric manifolds and show that they belong to four basic classes from the considered classification. We define an almost paracontact metric structure on any 3-dimensional Lie group and give concrete examples of Lie groups belonging to each of the four basic classes, characterized by commutators on the corresponding Lie algebras.
\end{abstract}

\newcommand{\g}{\mathfrak{g}}
\newcommand{\s}{\mathfrak{S}}
\newcommand{\D}{\mathcal{D}}
\newcommand{\F}{\mathcal{F}}
\newcommand{\R}{\mathbb{R}}
\newcommand{\K}{\mathbb{K}}
\newcommand{\U}{\mathbb{U}}
\newcommand{\diag}{\mathrm{diag}}
\newcommand{\End}{\mathrm{End}}
\newcommand{\im}{\mathrm{Im}}
\newcommand{\id}{\mathrm{id}}
\newcommand{\Hom}{\mathrm{Hom}}

\newcommand{\Rad}{\mathrm{Rad}}
\newcommand{\rank}{\mathrm{rank}}
\newcommand{\const}{\mathrm{const}}
\newcommand{\tr}{{\rm tr}}
\newcommand{\ltr}{\mathrm{ltr}}
\newcommand{\codim}{\mathrm{codim}}
\newcommand{\Ker}{\mathrm{Ker}}

\newcommand{\thmref}[1]{Theorem~\ref{#1}}
\newcommand{\propref}[1]{Proposition~\ref{#1}}
\newcommand{\corref}[1]{Corollary~\ref{#1}}
\newcommand{\secref}[1]{\S\ref{#1}}
\newcommand{\lemref}[1]{Lemma~\ref{#1}}
\newcommand{\dfnref}[1]{Definition~\ref{#1}}


\newcommand{\ee}{\end{equation}}
\newcommand{\be}[1]{\begin{equation}\label{#1}}

\maketitle

\section{Introduction}\label{sec-1}
Different manifolds with additional tensor structures have been classified with respect to the structure $(0,3)$ tensors, generated by the covariant derivative of the
fundamental tensor of type $(1,1)$. For example, such classifications are: the Gray-Hervella classification of almost Hermitian manifolds given in \cite{GH}, the Naveira classification of Riemannian almost product manifolds - in \cite{N}, the Ganchev-Borisov classification of almost complex manifolds with Norden metric - in
\cite{GB}, the Alexiev-Ganchev classification of almost contact metric manifolds - in \cite{AG}, the Ganchev-Mihova-Gribachev classification of almost contact B-metric manifolds - in \cite{GMG} and etc.
\par
In \cite{NZ} we decomposed the vector space of the structure $(0,3)$ tensors on almost paracontact metric manifolds (called almost paracontact manifolds with semi-Riemannian metric of $(n+1,n)$) in eleven subspaces which are orthogonal and invariant under the action of the structure group of the considered manifolds. In this paper we show that one of the eleven subspaces could be decomposed in two orthogonal and invariant subspaces and give their characteristic conditions. Then we find the
dimensions of the twelve subspaces and the projections of the structure tensor in the corresponding  basic classes of almost paracontact metric manifolds. Also, we obtain the classes of the following types of almost paracontact metric manifolds: $\alpha$-para-Sasakian,  $\alpha$-para-Kenmotsu, normal, paracontact metric, para-Sasakian, K-paracontact and quasi-para-Sasakian.
\par
We pay special attention to almost paracontact metric manifolds of dimension 3, which is the lowest dimension for these manifolds. First, we establish that such
manifolds belong only to four  basic classes from the considered classification. Then we define an almost paracontact metric structure on a 3-dimensional Lie group. We determine its Lie algebra by commutators such that the Lie group is a manifold belonging to some of the four basic classes of 3-dimensional almost paracontact metric manifolds. The considered Lie groups are characterized geometrically in terms of their curvature properties. Moreover, we find explicit matrix representations of these Lie groups. Let us note that Lie groups as 3-dimensional almost contact B-metric manifolds were studied in \cite{HMDM} and their matrix representations were obtained in
\cite{HM}.
\section{Preliminaries}\label{sec-2}
A (2n+1)-dimensional smooth manifold $M^{(2n+1)}$
has an \emph{almost paracontact structure} $(\varphi,\xi,\eta)$ if it admits a tensor field
$\varphi$ of type $(1,1)$, a vector field $\xi$ and a 1-form
$\eta$ satisfying the  following  conditions:
\begin{eqnarray}
  \label{f82}
    & &
    \begin{array}{cl}
     (i)   & \varphi^2 = id - \eta \otimes \xi, \quad \eta (\xi)=1, \quad \varphi(\xi)=0,
     \\[5pt]
     (ii) & \textrm{there exists a  distribution $\mathbb {D}: p \in M \longrightarrow \mathbb {D}_p\subset T_pM:$}
     \\[1pt]
     & \textrm{$\mathbb D_p=Ker \eta=\{x\in T_pM: \eta (x)=0\}$,  called {\it paracontact}}
     \\[1pt]
     & \textrm{{\it distribution} generated by $\eta$.}
    \end{array}
\end{eqnarray}
Then the tangent space $T_pM$ at each $p\in M$ is the following orthogonal direct sum
\[
T_pM=\mathbb D_p\oplus span_\mathbb R\{\xi (p)\}
\]
and every vector $x\in T_pM$ can be decomposed uniquely in the manner
\begin{equation}\label{3}
x=hx+vx,
\end{equation}
where $hx=\varphi ^2x \in \mathbb{D}_p$ and $vx=\eta (x).\xi (p) \in span_\mathbb R\{\xi (p)\}$. Using the conditions \eqref{f82} we have
\[
h\xi =0, \quad h^2=h, \quad h\circ \varphi =\varphi \circ h=\varphi ,\quad v\circ h=h\circ v =0.
\]
\par
The tensor field $\varphi $ induces an almost paracomplex structure \cite{KW} on each
fibre on $\mathbb D$ and $(\mathbb D, \varphi , g_{\vert \mathbb D})$ is a $2n$-dimensional
almost paracomplex manifold. Since $g$ is non-degenerate metric on $M$ and $\xi $ is non-isotropic,
the paracontact distribution $\mathbb D$ is non-degenerate.

As immediate consequences of the definition of the almost
paracontact structure we have that the endomorphism $\varphi$ has rank
$2n$,
and $\eta \circ \varphi=0$, (see \cite{B1,B2} for the almost contact case).\\
From now on, we will use  $x, y, z$ for arbitrary elements of $\chi (M)$ or vectors in the tangent space $T_pM$ at $p\in  M$.\\

If a manifold $M^{(2n+1)}$ with $(\varphi,\xi,\eta)$-structure
admits a pseudo-Riemannian metric $g$ such that
\begin{equation}\label{con}
g(\varphi x,\varphi y)=-g(x,y)+\eta (x)\eta (y),
\end{equation}
then we say that $M^{(2n+1)}$ has an almost paracontact metric structure and
$g$ is called \emph{compatible} metric. Any compatible metric $g$ with a given almost paracontact
structure is necessarily of signature $(n+1,n)$.

Setting $y=\xi$, we have
$\eta(x)=g(x,\xi).$

Any almost paracontact structure admits a compatible metric.

The fundamental 2-form
\begin{equation}\label{fund}
\phi(x,y)=g(\varphi x,y)
\end{equation}
is non-degenerate on the horizontal distribution $\mathbb D$ and
$\eta\wedge {\phi}^n\not=0$.
\par
Let $\phi$ be the fundamental 2-form on $(M,\varphi ,\xi ,\eta ,g)$
and $F$ be the covariant derivative of $\phi$ with respect to the
Levi-Civita connection $\nabla $ of $g$, i.e the
tensor field $F$ of type $(0,3)$ is defined by
\begin{equation}\label{4}
F(x,y,z)=(\nabla \phi )(x,y,z)=(\nabla _x\phi )(y,z)=\\
g((\nabla _x\varphi )y,z).
\end{equation}
Because of $(1.1)$ and $(1.3)$ the tensor $F$ has the following properties:
\begin{equation}\label{5}
\begin{array}{ll}
F(x,y,z)=-F(x,z,y), \\
F(x,\varphi y, \varphi z)=F(x,y,z)+\eta(y)F(x,z,\xi)-\eta(z)F(x,y,\xi).
\end{array}
\end{equation}
The following 1-forms are associated with $F$:
\begin{equation}\label{6}
\theta(x)=g^{ij}F(e_i,e_j,x); \,
\theta^*(x)=g^{ij}F(e_i,\varphi e_j,x); \,
\omega(x)=F(\xi,\xi,x),
\end{equation}
where $\{e_i,\xi\}$ $(i=1,\ldots,2n)$ is a basis of $TM$, and $(g^{ij})$ is the inverse matrix of $(g_{ij})$.

We express $\nabla \eta$, $d\eta $, $L_\xi g$ and $d\phi $ in terms of the structure tensor $F$ in the following lemma
\begin{lem}\label{Lemma 2.1}
For arbitrary $x, y, z$ we have:
\begin{equation}\label{2.2}
(\nabla _x\eta)y=g(\nabla _x\xi ,y)=-F(x,\varphi y,\xi );
\end{equation}
\begin{equation}\label{2.3}
d\eta (x,y)=\frac{1}{2}\left((\nabla _x\eta)y-(\nabla _y\eta)x\right)=\frac{1}{2}(-F(x,\varphi y,\xi )+F(y,\varphi x,\xi ));
\end{equation}
\begin{equation}\label{2.4}
(L_\xi g)(x,y)=(\nabla _x\eta)y+(\nabla _y\eta)x=-F(x,\varphi y,\xi )-F(y,\varphi x,\xi );
\end{equation}
\begin{equation}\label{2.5}
d\phi (x,y,z)=\mathop{\s} \limits_{(x,y,z)}F(x,y,z) ,
\end{equation}
where $\mathop{\s} \limits_{(x,y,z)}$ denotes the cyclic sum over $x, y, z$.
\end{lem}
In \cite{JW} it is proved that a $(2n+1)$-dimensional almost paracontact metric manifold is normal if and only if the following condition holds:
\begin{equation}\label{2.6}
\varphi (\nabla _x\varphi)y-(\nabla _{\varphi x}\varphi)y+(\nabla _x\eta)(y)\xi =0 .
\end{equation}
Moreover, we have that \eqref{2.6} is equivalent to the following equality
\begin{equation}\label{2.7}
F(x, y,\varphi z)+F(\varphi x,y,z)+F(x,\varphi y,\xi )=0 .
\end{equation}

\begin{defn}\label{Definition 2.2}
A $(2n+1)$-dimensional almost paracontact metric manifold is called
\begin{itemize}
\item {\it normal} if $N(x,y)-2d\eta (x,y)\xi = 0$, where
\[
N(x,y)=\varphi ^2[x,y]+[\varphi x,\varphi y]-\varphi [\varphi x,y]-\varphi [x,\varphi y]
\]
is the Nijenhuis torsion tensor of $\varphi $ (see \cite{Z});
\item {\it paracontact metric} if $\phi =d\eta$;
\item  {\it $\alpha $-para-Sasakian} if $(\nabla_x\varphi)y=\alpha(g(x,y)\xi-\eta(y)x)$, where $\alpha\neq 0$ is constant; \item {\it para-Sasakian} if it is normal and paracontact metric;
\item  {\it $\alpha $-para-Kenmotsu} if $(\nabla_x\varphi)y=-\alpha(g(x,\varphi y)\xi+\eta(y)\varphi x)$, where $\alpha\neq 0$ is constant, in particular, para-Kenmotsu if $\alpha=-1$;
\item {\it K-paracontact} if it is paracontact and $\xi$ is Killing vector field;
\item {\it quasi-para-Sasakian} if it is normal and $d\phi =0$.
\end{itemize}
\end{defn}

\begin{rem}\label{Remark 1.}
In \cite{Z} it was proved that $(M,\varphi ,\xi ,\eta ,g)$ is para-Sasakian if and only if $(\nabla_x\varphi)y=-g(x,y)\xi+\eta(y)x$. This result is obtained by $\phi (x,y)=g(x,\varphi y)$. We note that if
$\phi (x,y)=g(\varphi x,y)$, then $(M,\varphi ,\xi ,\eta ,g)$ is para-Sasakian if and only if $(\nabla_x\varphi)y=g(x,y)\xi-\eta(y)x$. Hence, if $\phi (x,y)=g(x,\varphi y)$ (resp.  $\phi (x,y)=g(\varphi x,y)$), then an $\alpha $-para-Sasakian manifold is para-Sasakian if $\alpha=-1$ (resp. $\alpha=1$).
\end{rem}




Let $\mathbb U^ \pi(n)$ be the paraunitary group, i.e. $\mathbb U^ \pi(n)$ consists of paracomplex
matrices $\beta =A+\epsilon B$ \, ($\epsilon ^2=1$; \, $A, B$ are real matrices of type $(n\times n)$) such
that $\beta ^{-1}=\bar \beta ^t$. If $r$ is the real representation of $\mathbb U^ \pi(n)$ then
\[
r(\beta )=\left(\begin {matrix} A & B
\cr B & A\cr \end {matrix}\right) , \quad A^tA-B^tB=I_n, \quad A^tB-B^tA=0 ,
\]
where $\beta \in \mathbb U^ \pi(n)$, $I_n$ denotes the identity matrix of type $(n\times n)$.
We consider the group $\mathbb U^ \pi(n)\times \{1\}$ which consists of matrices $\alpha $ of type
$((n+1)\times (n+1))$ such that
$\alpha =\left(\begin {matrix} &  0  \cr  \beta &  \vdots \cr &  0  \cr 0 \ldots 0 & 1
\end {matrix}\right) , \beta \in \mathbb U^ \pi(n)$. Then $r(\alpha )=
\left(\begin {matrix}A & B & 0  \cr  B  &  A & \vdots   \cr & &  0  \cr  0 &  \ldots 0 & 1
\end {matrix}\right)$. \\
For $\alpha \in \mathbb U^ \pi(n)\times id$ we have $\alpha \xi=\xi , \, \alpha \circ \varphi =
\varphi \circ \alpha $ and $g$ is an isometry with respect to $\alpha$, i.e. the matrices of
$\mathbb U^ \pi(n)\times \{1\}$ preserve the structures $\xi , \varphi , g, \eta$. Hence,
$\mathbb U^ \pi(n)\times \{1\}$ is the structure group of the almost paracontact metric manifolds.
\par
Let $V$ be a $(2n+1)$-dimensional real vector space with an almost paracontact structure $(\varphi,\xi,\eta)$ and a compatible metric $g$ with this structure.
We denote by $\otimes ^0_3V$ the space of the tensors of type $(0,3)$ over $V$. Let $\mathcal{F}$ be the subspace of $\otimes ^0_3V$ defined by
\begin{equation*}
\begin{array}{lr}
\mathcal{F}=\{F\in \otimes ^0_3V : F(x,y,z)=-F(x,z,y)=F(x,\varphi y,\varphi z)-\eta (y)F(x,z,\xi ) \\
\qquad \qquad \qquad \qquad \qquad \qquad +\eta (z)F(x,y,\xi )\}.
\end{array}
\end{equation*}
The metric $g$ on $V$ induces an inner product $\langle , \rangle$ on $\mathcal{F}$ which is defined by
\[
\langle F_1,F_2\rangle=g^{ip}g^{jq}g^{kr}F_1(f_i,f_j,F_k)F_2(f_p,f_q,F_r) ,
\]
where $F_1, F_2 \in \mathcal{F}$ and $\{f_1,\ldots ,f_{2n}\}$ is a basis of $V$.
\par
The standard representation of the structure group $\mathbb U^ \pi(n)\times \{1\}$ in $V$ induces a representation $\lambda $ of $\mathbb U^ \pi(n)\times \{1\}$ in $\mathcal{F}$ in the following manner:
\[
(\lambda (\alpha )F)(x,y,z)=F(\alpha ^{-1}x,\alpha ^{-1}y,\alpha ^{-1}z) ,
\]
for $\alpha \in \mathbb U^ \pi(n)\times \{1\}$ and $F\in \mathcal{F}$. Also, $\lambda (\alpha )$
preserves the inner product $\langle , \rangle$ in $\mathcal{F}$.

\section{On the decomposition of $\mathcal{F}$}\label{sec-7}
In \cite{NZ} we obtained a decomposition of a vector space $\mathcal{F}$ into eleven subspaces
$\mathcal{F}_i$ $(i =1, \ldots ,11 )$, which are mutually orthogonal and
invariant under the action of the structure group $\mathbb {U}^\pi (n)\times \{1\}$. First we found the
following partial decomposition of $\mathcal{F}$ in a direct sum of its subspaces $W_i$ $(i =1,2,3,4 )$, i.e.
\[
\mathcal{F}=W_1\oplus W_2\oplus W_3\oplus W_4 ,
\]
where $W_i$ $(i =1,2,3,4 )$ were defined by
\begin{equation}\label {7.1}
\begin{array}{llll}
W_1=\{F\in \mathcal{F} : F(x,y,z)=F(hx,hy,hz)\} , \\
W_2=\{F\in \mathcal{F} : F(x,y,z)=-\eta (y)F(hx,hz,\xi )+\eta (z)F(hx,hy,\xi )\} , \\
W_3=\{F\in \mathcal{F} : F(x,y,z)=\eta (x)F(\xi ,hy,hz) , \\
W_4=\{F\in \mathcal{F} : F(x,y,z)=\eta (x)\{\eta (y)F(\xi ,\xi ,hz)-\eta (z)F(\xi ,\xi ,hy)\} ,
\end{array}
\end{equation}
for arbitrary vectors $x,y,z \in V$. The subspaces $W_i$ $(i =1,2,3,4 )$ are mutually orthogonal and invariant under the action of $\mathbb {U}^\pi (n)\times \{1\}$.

\subsection{The subspace $W_1$ of $\mathcal{F}$}\label{subsec-7.1}
In \cite{NZ} we obtained  that $W_1=\mathcal{F}_1\oplus \mathcal{F}_2\oplus \mathcal{F}_3$, where the subspaces $\mathcal{F}_i$ $(i=1,2,3)$  of $W_1$ are mutually orthogonal and invariant under the action of $\mathbb {U}^\pi (n)\times \{1\}$.  They were characterized by
\begin{equation}\label{7.2}
\begin{array}{ll}
\mathcal{F}_1=\{F\in \mathcal{F} : F(x,y,z)=\frac{1}{2(n-1)}\{g(x,\varphi y)\theta _F(\varphi z)-
g(x,\varphi z)\theta _F(\varphi y)\\ \\
\qquad \qquad \qquad \qquad \qquad -g(\varphi x,\varphi y)\theta _F(hz)+g(\varphi x,\varphi z)\theta _F(hy)\}\} ,
\end{array}
\end{equation}
\begin{equation}\label{7.3}
\begin{array}{l}
\mathcal{F}_2=\{F\in \mathcal{F} : F(\varphi x ,\varphi y,z)=-F(x,y,z) , \quad \theta _F=0\} ,
\end{array}
\end{equation}

\begin{equation}\label{7.4}
\begin{array}{l}
\mathcal{F}_3=\{F\in \mathcal{F} : F(\varphi x ,\varphi y,z)=F(x,y,z)\} .
\end{array}
\end{equation}
Now, we will show that the subspace $\mathcal{F}_3$ could be decomposed in an orthogonal direct sum of two its subspaces. For this purpose we define the following linear map
\[
k : \mathcal{F}_3\longrightarrow  \mathcal{F}_3  \quad {\text by} \quad
k(F)(x,y,z)=\frac{1}{3}\{F(x,y,z)+F(y,z,x)+F(z,x,y)\}.
\]
We note that from \eqref{7.4} it follows  that $F(\xi ,y,z)=F(x,\xi ,z)=0$ for an arbitrary $F\in  \mathcal{F}_3$. Then one can easily verify that $k(F)(\varphi x ,\varphi y,z)=k(F)(x,y,z)$, i.e.
$k(F)$ belongs to $\mathcal{F}_3$. By direct computations we check that $k$ is a projection
(i.e. $k^2=k$) and it commutes with the action of $\mathbb {U}^\pi (n)\times \{1\}$. We put
\begin{equation}\label{7.5}
\begin{array}{l}
\mathbb{G}_3={\rm Im} k=\{F\in \mathcal{F}_3 : F(x,y,z)=\frac{1}{3}\{\mathop{\s} \limits_{(x,y,z)}F(x,y,z)\}\} ,
\end{array}
\end{equation}
\begin{equation}\label{7.6}
\begin{array}{l}
\mathbb{G}_4={\rm Ker} k=\{F\in \mathcal{F}_3 : \mathop{\s} \limits_{(x,y,z)}F(x,y,z)=0\} .
\end{array}
\end{equation}
\begin{pro}\label{Proposition 7.1}
The subspace $\mathcal{F}_3$ is an orthogonal direct sum of the subspaces $\mathbb{G}_3$ and
$\mathbb{G}_4$. These subspaces are invariant under the action of $\mathbb {U}^\pi (n)\times \{1\}$.
\end{pro}
\begin{proof}
Taking into account that $k$ is a projection in $\mathcal{F}_3$, we have $\mathcal{F}_3={\rm Im} k\oplus {\rm Ker} k=\mathbb{G}_3\oplus \mathbb{G}_4$.  Further, we will show that $\mathbb{G}_3$ and $\mathbb{G}_4$ are orthogonal.
Because for an arbitrary $F\in \mathcal{F}$ the conditions $F(x,y,z)=-F(y,x,z)$ and $F(x,y,z)=\frac{1}{3}\{F(x,y,z)+F(y,z,x)+F(z,x,y)\}$ are equivalent, \eqref{7.5} becomes
\begin{equation}\label{7.7}
\begin{array}{l}
\mathbb{G}_3={\rm Im} k=\{F\in \mathcal{F}_3 : F(x,y,z)=-F(y,x,z)\} .
\end{array}
\end{equation}
Now, we take $F^\prime \in \mathbb{G}_3$ and $F'' \in \mathbb{G}_4$. Using \eqref{7.6} and
\eqref{7.7} we obtain
\begin{equation*}
\begin{array}{lll}
\langle F^\prime ,F''\rangle=g^{ip}g^{jq}g^{kr}F_1(f_i,f_j,F_k)F_2(f_p,f_q,F_r)=
-g^{ip}g^{jq}g^{kr}F^\prime (f_i,f_j,F_k)F''(f_q,f_r,F_p) \\ \\
-g^{ip}g^{jq}g^{kr}F^\prime (f_i,f_j,F_k)F''(f_r,f_p,F_q)=
-g^{jq}g^{ip}g^{kr}F^\prime (f_j,f_i,F_k)F''(f_q,f_p,F_r) \\ \\
-g^{kr}g^{ip}g^{jq}F^\prime (f_k,f_i,F_j)F''(f_r,f_p,F_q)=-2\langle F^\prime ,F''\rangle .
\end{array}
\end{equation*}
Thus we find $\langle F^\prime ,F''\rangle=0$. Hence, $\mathbb{G}_3$ and $\mathbb{G}_4$ are orthogonal.
\par
Finally, taking into account that $k$ commutes with  the action of $\mathbb {U}^\pi (n)\times \{1\}$, for an arbitrary $F^\prime \in \mathbb{G}_3={\rm Im} k$ we have
\begin{equation}\label{7.8}
\lambda (\alpha )(F^\prime )=
\lambda (\alpha )(k(F^\prime ))=k(\lambda (\alpha )(F^\prime )).
\end{equation}
We note that $\lambda (\alpha )(F^\prime )\in \mathcal{F}_3$ because $\mathcal{F}_3$ is invariant under the action of $\mathbb {U}^\pi (n)\times \{1\}$. Then from \eqref{7.8} it follows that
$\lambda (\alpha )(F^\prime )\in \mathbb{G}_3$ which means that $\mathbb{G}_3$ is invariant
under the action of $\mathbb {U}^\pi (n)\times \{1\}$. Since $\lambda (\alpha )$ is an isometry with
respect to the inner product $\langle , \rangle$ in $\mathcal{F}$, the orthogonal complement
$\mathbb{G}_4$ of the invariant subspace $\mathbb{G}_3$ in $\mathcal{F}_3$  is also invariant.
\end{proof}
From now on, we will denote the subspaces $\mathcal{F}_1$ and $\mathcal{F}_2$ by $\mathbb{G}_1$ and $\mathbb{G}_2$ , respectively. In conclusion, we state
\begin{pro}\label{Proposition 7.2}
The decomposition $W_1=\mathbb{G}_1\oplus \mathbb{G}_2\oplus \mathbb{G}_3\oplus \mathbb{G}_4$ is orthogonal and invariant under the action of $\mathbb {U}^\pi (n)\times \{1\}$.
\end{pro}
The characteristic conditions of $\mathbb{G}_1$ and $\mathbb{G}_2$ are \eqref{7.2} and  \eqref{7.3},
respectively, which were obtained in \cite{NZ}. According to \eqref{7.7} and \eqref{7.6}, the characteristic conditions of $\mathbb{G}_3$ and $\mathbb{G}_4$ are as follows:
\begin{equation}\label{7.9}
\mathbb{G}_3=\{F\in \mathcal{F} : F(\xi ,y,z)=F(x,\xi ,z)=0, \quad   F(x,y,z)=-F(y,x,z)\} ,
\end{equation}
\begin{equation}\label{7.10}
\mathbb{G}_4=\{F\in \mathcal{F} : F(\xi ,y,z)=F(x,\xi ,z)=0, \quad   \mathop{\s} \limits_{(x,y,z)}F(x,y,z)=0\} .
\end{equation}

\subsection{The subspace $W_2$ of $\mathcal{F}$}\label{subsec-7.2}
In \cite{NZ} we decomposed $W_2$ into 6 subspaces which are mutually orthogonal and invariant under  the action of $\mathbb {U}^\pi (n)\times \{1\}$, i.e.
\[
W_2=\mathcal{F}_4\oplus \mathcal{F}_5\oplus \mathcal{F}_6\oplus \mathcal{F}_7\oplus
\mathcal{F}_8\oplus \mathcal{F}_9 .
\]
Taking into account the characteristic condition of $W_2$ in \eqref{7.1}, we rewrite the conditions of
$\mathcal{F}_6, \mathcal{F}_7, \mathcal{F}_8$ and $\mathcal{F}_9$ in an equivalent form to the one in \cite[Theorem 2.1, p. 124]{NZ}. Moreover, we denote the subspaces $\mathcal{F}_4, \mathcal{F}_5, \mathcal{F}_6, \mathcal{F}_7, \mathcal{F}_8, \mathcal{F}_9$ by $\mathbb{G}_5, \mathbb{G}_6, \mathbb{G}_7, \mathbb{G}_8, \mathbb{G}_9$, $\mathbb{G}_{10}$, respectively. So, we have:
\begin{equation*}
\begin{array}{llllll}
\mathbb{G}_5=\mathcal{F}_4=\left\{F\in \mathcal{F} : F(x,y,z)=\displaystyle{\frac{\theta _F(\xi)}{2n}}\{\eta(y)g(\varphi x,\varphi z)-\eta(z)g(\varphi x,\varphi y) \}\right\} . \\ \\

$This is the class of generalized $\alpha$-para-Sasakian manifolds.$ \\ \\

\mathbb{G}_6=\mathcal{F}_5=\left\{F\in \mathcal{F} : F(x,y,z)=-\displaystyle{\frac{\theta _F^*(\xi)}{2n}}\{\eta(y)g(x,\varphi z)-\eta(z)g(x,\varphi y)\}\right\} . \\ \\

$This is the class of generalized $\alpha$-para-Kenmotsu manifolds.$ \\ \\

\mathbb{G}_7=\mathcal{F}_6=\left\{F\in \mathcal{F} : F(x,y,z)=-\eta(y)F(x,z,\xi )+\eta(z)F(x,y,\xi ),
\right .\\  \\
\qquad \qquad \quad \left .F(x,y,\xi )=-F(y,x,\xi )=-F(\varphi x,\varphi y,\xi ), \quad \theta _F^*(\xi)=0\right\} , \\ \\
\mathbb{G}_8=\mathcal{F}_7=\left\{F\in \mathcal{F} : F(x,y,z)=-\eta(y)F(x,z,\xi )+\eta(z)F(x,y,\xi ),
\right .\\  \\
\qquad \qquad \quad \left .F(x,y,\xi )=F(y,x,\xi )=-F(\varphi x,\varphi y,\xi ), \quad \theta _F(\xi)=0\right\} ,
\end{array}
\end{equation*}
\begin{equation*}
\begin{array}{llll}
\mathbb{G}_9=\mathcal{F}_8=\left\{F\in \mathcal{F} : F(x,y,z)=-\eta(y)F(x,z,\xi )+\eta(z)F(x,y,\xi ),
\right .\\  \\
\qquad \qquad \quad \left .F(x,y,\xi )=-F(y,x,\xi )=F(\varphi x,\varphi y,\xi )\right\} , \\ \\
\mathbb{G}_{10}=\mathcal{F}_9=\left\{F\in \mathcal{F} : F(x,y,z)=-\eta(y)F(x,z,\xi )+\eta(z)F(x,y,\xi ),
\right .\\  \\
\qquad \qquad \quad \left .F(x,y,\xi )=F(y,x,\xi )=F(\varphi x,\varphi y,\xi )\right\}.
\end{array}
\end{equation*}

\begin{rem}\label{Remark 7.1}
We call the classes $\mathbb{G}_5$ and $\mathbb{G}_6$ the class of generalized $\alpha$-para-Sasakian and generalized $\alpha$-para-Kenmotsu manifolds, respectively, because $\theta _F(\xi)$ and $\theta _F^*(\xi)$ are functions in general.
\end{rem}

\subsection{The subspaces $W_3$ and $W_4$ of $\mathcal{F}$}\label{subsec-7.3}
As in \cite{NZ} we put $\mathcal{F}_{10}=W_3$ and $\mathcal{F}_{11}=W_4$. Now, we denote
$\mathcal{F}_{10}$ and $\mathcal{F}_{11}$ by $\mathbb{G}_{11}$ and $\mathbb{G}_{12}$, respectively.
These subspaces were represented by
\begin{equation*}
\begin{array}{ll}
\mathbb{G}_{11}=\mathcal{F}_{10}=\left\{F\in \mathcal{F} : F(x,y,z)=\eta(x)F(\xi ,\varphi y,\varphi z)\right\} , \\ \\
\mathbb{G}_{12}=\mathcal{F}_{11}=\left\{F\in \mathcal{F} : F(x,y,z)=\eta(x)\left\{\eta(y)F(\xi ,\xi ,z)-\eta(z)F(\xi ,\xi ,y)\right\}\right\} .
\end{array}
\end{equation*}
\par
Corresponding to the decomposition in \secref{sec-7} of the space $ \mathcal{F}$  into 12 mutually orthogonal and invariant subspaces, we give 12 classes of almost paracontact metric manifolds. An almost paracontact metric manifold $M$ is said to be in the class $\mathbb{G}_i$ $(i=1,\ldots ,12)$ (or $\mathbb{G}_i$-manifold) if at each $p\in M$ the tensor $F$ of $M$ belongs to the subspace $\mathbb{G}_i$. The special class $\mathbb{G}_0$, determined by the condition $F(x,y,z) = 0$, is the intersection of the basic twelve classes. Hence,  $\mathbb{G}_0$ is the class of the almost paracontact metric manifolds with parallel structures, i.e. $\nabla \varphi =\nabla \xi =\nabla \eta =\nabla g=0$.
\par
Finally, by using  the characteristic symmetries  of the structure tensor $F$ of a $(2n+1)$-dimensional almost paracontact metric manifold in the classes $\mathbb{G}_i$ $(i=1,\ldots ,12)$, we obtain
\begin{thm}\label{Theorem 2.1}
The dimensions of the subspaces in the decomposition of the space ${\F}$ are as follows:
\begin{align*}
\begin{array}{lll}
{\rm dim} \, \mathbb{G}_1=2(n-1), & {\rm dim} \, \mathbb{G}_2=(n-1)(n^2-2), & {\rm dim} \, \mathbb{G}_3=\displaystyle\frac{(n-2)(n-1)n}{3}\\ \\
\mathbb{G}_4=\displaystyle\frac{2(n-1)n(n+1)}{3}, & {\rm dim} \, \mathbb{G}_5=1, &
{\rm dim} \, \mathbb{G}_6=1, \\
{\rm dim} \, \mathbb{G}_7=n^2-1, & {\rm dim} \, \mathbb{G}_8=n^2-1, &
{\rm dim} \, \mathbb{G}_9=n(n-1),  \\
{\rm dim} \, \mathbb{G}_{10}=n(n+1), & {\rm dim} \, \mathbb{G}_{11}=n(n-1),
& {\rm dim} \, \mathbb{G}_{12}=2n .
\end{array}
\end{align*}
\end{thm}

\section{The projections of the structure tensor $F$ in the twelve basic classes of the classification of the almost paracontact metric manifolds}\label{sec-3}

The decompositions of $\mathcal{F}$ in   direct sums  of the subspaces $W_j$ $(j=1,2,3,4)$ and
$\mathbb{G}_i$  $(i=1,\ldots ,12)$ imply that every $F\in \mathcal{F}$ has a unique representation in the form $F(x,y,z)=\sum \limits _{j=1}^{4} F^{W_j}(x,y,z)$ and $F(x,y,z)=
\sum \limits _{i=1}^{12} F^i(x,y,z)$, respectively, where $F^{W_j}\in W_j$ and
$F^i\in \mathbb{G}_i$.
Then it is clear that an almost paracontact metric manifold $(M,\varphi ,\xi ,\eta ,g )$ belongs to a direct sum of two or more basic classes, i.e. $M\in \mathbb{G}_i\oplus \mathbb{G}_j\oplus \ldots $, if and only if the structure tensor $F$ on $M$ is the sum of the corresponding projections $F^i$, $F^j$, $\ldots $, i.e. the following condition is satisfied $F =F^i+F^j+\ldots $. \\
Following the operators defined in \cite[Theorem 2.1, p. 124]{NZ} and the operator $k$ defined in this paper in  \secref{subsec-7.1}, we find the projections $F^{W_j}$ $(j=1,2,3,4)$ and $F^i$ $(i =1, \ldots ,12 )$ of $F\in {\F}$ in the subspaces $W_j$ and $\mathbb{G}_i$, respectively.
These projections are given below:
\begin{equation}\label{0)}
\begin{array}{l}
F^{W_1}=F(\varphi ^2x,\varphi ^2y,\varphi ^2z) ,\\
F^{W_2}=-\eta(y)F(\varphi ^2x,\varphi ^2z,\xi )+\eta(z)F(\varphi ^2x,\varphi ^2y,\xi ) ,\\
F^{W_3}=\eta(x)F(\xi ,\varphi y,\varphi z) , \\
F^{W_4}=\eta(x)\{\eta (y)F(\xi ,\xi ,z)-\eta (z)F(\xi ,\xi ,y)\} .
\end{array}
\end{equation}

\begin{equation}\label{1)}
\begin{array}{l}
F^1(x,y,z)=\displaystyle{\frac{1}{2(n-1)}}\left\{g(x,\varphi y)\theta _{F^1}(\varphi z)-g(x,\varphi z)\theta _{F^1}(\varphi y)\right. \\ \\
\left. \qquad \qquad \quad -g(\varphi x,\varphi y)\theta _{F^1}(\varphi ^2z)+g(\varphi x,\varphi z)\theta _{F^1}(\varphi ^2y)\right\};
\end{array}
\end{equation}
\begin{equation}\label{2)}
\begin{array}{l}
F^2(x,y,z)=\displaystyle{\frac{1}{2}}\left\{F(\varphi ^2x,\varphi ^2y,\varphi ^2z)-F(\varphi x,\varphi ^2y,\varphi z)\right\}\\ \\
\qquad \qquad \qquad -\displaystyle{\frac{1}{2(n-1)}}\left\{g(x,\varphi y)\theta _{F^1}(\varphi z)-g(x,\varphi z)\theta _{F^1}(\varphi y)\right. \\ \\
\left. \qquad \qquad \qquad -g(\varphi x,\varphi y)\theta _{F^1}(\varphi ^2z)+g(\varphi x,\varphi z)\theta _{F^1}(\varphi ^2y)\right\};
\end{array}
\end{equation}
\begin{equation}\label{3')}
\begin{array}{lll}
F^3(x,y,z)=\displaystyle{\frac{1}{6}}\left\{F(\varphi ^2x,\varphi ^2y,\varphi ^2z)+F(\varphi x,\varphi ^2y,\varphi z)\right. \\ \\
\left.+F(\varphi ^2y,\varphi ^2z,\varphi ^2x)+F(\varphi y,\varphi ^2z,\varphi x)\right. \\ \\
\left.+F(\varphi ^2z,\varphi ^2x,\varphi ^2y)+F(\varphi z,\varphi ^2x,\varphi y)\right\};
\end{array}
\end{equation}
\begin{equation}\label{3'')}
\begin{array}{llll}
F^4(x,y,z)=\displaystyle{\frac{1}{2}}\left\{F(\varphi ^2x,\varphi ^2y,\varphi ^2z)+F(\varphi x,\varphi ^2y,\varphi z)\right\} \\ \\
-\displaystyle{\frac{1}{6}}\left\{F(\varphi ^2x,\varphi ^2y,\varphi ^2z)+F(\varphi x,\varphi ^2y,\varphi z)\right. \\ \\
\left.+F(\varphi ^2y,\varphi ^2z,\varphi ^2x)+F(\varphi y,\varphi ^2z,\varphi x)\right. \\ \\
\left.+F(\varphi ^2z,\varphi ^2x,\varphi ^2y)+F(\varphi z,\varphi ^2x,\varphi y)\right\};
\end{array}
\end{equation}
\begin{equation}\label{4)}
\begin{array}{l}
F^5(x,y,z)=\displaystyle{\frac{\theta _{F^5}(\xi)}{2n}}\{\eta(y)g(\varphi x,\varphi z)-
\eta(z)g(\varphi x,\varphi y) \};
\end{array}
\end{equation}
\begin{equation}\label{5)}
\begin{array}{l}
F^6(x,y,z)=-\displaystyle{\frac{\theta _{F^6}^*(\xi)}{2n}}\{\eta(y)g(x,\varphi z)-
\eta(z)g(x,\varphi y) \};
\end{array}
\end{equation}
\begin{equation}\label{6)}
\begin{array}{ll}
F^7(x,y,z)=\displaystyle{-\frac{1}{4}}\eta (y)\left\{F(\varphi ^2x,\varphi ^2z,\xi )-F(\varphi x,\varphi z,\xi )\right.\\ \\
\left.-F(\varphi ^2z,\varphi ^2x,\xi )+F(\varphi z,\varphi x,\xi )\right\}+
\displaystyle{\frac{1}{4}}\eta (z)\left\{F(\varphi ^2x,\varphi ^2y,\xi )\right.\\ \\
\left.-F(\varphi x,\varphi y,\xi )-F(\varphi ^2y,\varphi ^2x,\xi )+F(\varphi y,\varphi x,\xi )\right\}\\ \\
\qquad \qquad \qquad+\displaystyle{\frac{\theta _{F^6}^*(\xi)}{2n}}\{\eta(y)g(x,\varphi z)-
\eta(z)g(x,\varphi y) \};
\end{array}
\end{equation}
\begin{equation}\label{7)}
\begin{array}{l}
F^8(x,y,z)=\displaystyle{-\frac{1}{4}}\eta (y)\left\{F(\varphi ^2x,\varphi ^2z,\xi )-F(\varphi x,\varphi z,\xi )\right.\\ \\
\left.+F(\varphi ^2z,\varphi ^2x,\xi )-F(\varphi z,\varphi x,\xi )\right\}+
\displaystyle{\frac{1}{4}}\eta (z)\left\{F(\varphi ^2x,\varphi ^2y,\xi )\right.\\ \\
\left.-F(\varphi x,\varphi y,\xi )+F(\varphi ^2y,\varphi ^2x,\xi )-F(\varphi y,\varphi x,\xi )\right\}\\ \\
\qquad \qquad \qquad-\displaystyle{\frac{\theta _{F^5}(\xi)}{2n}}\{\eta(y)g(\varphi x,\varphi z)-
\eta(z)g(\varphi x,\varphi y) \};
\end{array}
\end{equation}
\begin{equation}\label{8)}
\begin{array}{l}
F^9(x,y,z)=\displaystyle{-\frac{1}{4}}\eta (y)\left\{F(\varphi ^2x,\varphi ^2z,\xi )+F(\varphi x,\varphi z,\xi )\right.\\ \\
\left.-F(\varphi ^2z,\varphi ^2x,\xi )-F(\varphi z,\varphi x,\xi )\right\}+
\displaystyle{\frac{1}{4}}\eta (z)\left\{F(\varphi ^2x,\varphi ^2y,\xi )\right.\\ \\
\left.+F(\varphi x,\varphi y,\xi )-F(\varphi ^2y,\varphi ^2x,\xi )-F(\varphi y,\varphi x,\xi )\right\};
\end{array}
\end{equation}
\begin{equation}\label{9)}
\begin{array}{l}
F^{10}(x,y,z)=\displaystyle{-\frac{1}{4}}\eta (y)\left\{F(\varphi ^2x,\varphi ^2z,\xi )+F(\varphi x,\varphi z,\xi )\right.\\ \\
\left.+F(\varphi ^2z,\varphi ^2x,\xi )+F(\varphi z,\varphi x,\xi )\right\}+
\displaystyle{\frac{1}{4}}\eta (z)\left\{F(\varphi ^2x,\varphi ^2y,\xi )\right.\\ \\
\left.+F(\varphi x,\varphi y,\xi )+F(\varphi ^2y,\varphi ^2x,\xi )+F(\varphi y,\varphi x,\xi )\right\};
\end{array}
\end{equation}
\begin{equation}\label{10)}
\begin{array}{l}
F^{11}(x,y,z)=\eta(x)F(\xi ,\varphi ^2y,\varphi ^2z);
\end{array}
\end{equation}
\begin{equation}\label{11)}
F^{12}(x,y,z)=\eta(x)\left\{\eta(y)F(\xi ,\xi ,\varphi ^2z)-\eta(z)F(\xi ,\xi ,\varphi ^2y)\right\}.
\end{equation}

Further in this section, using the characteristic conditions of the twelve classes of almost paracontact metric manifolds and the projections of the structure tensor $F$, we relate the obtained classes with those studied in the literature.
\begin{thm}\label{Theorem 2.2}
A $(2n+1)$-dimensional almost paracontact metric manifold $M(\varphi ,\xi ,\\
\eta ,g)$ is normal if and only if $M$ belongs to one of the classes $\mathbb{G}_1$, $\mathbb{G}_2$, $\mathbb{G}_5$, $\mathbb{G}_6$, $\mathbb{G}_7$, $\mathbb{G}_8$ or to the classes which are their direct sums.
\end{thm}
\begin{proof}
Let $M$ belongs to $\mathbb{G}_i$ $(i=1,2,5,6,7,8)$ or to their direct sums. By direct computations we check that for the structure tensor $F$ of $M$ the condition \eqref{2.7} holds. Hence, $M$ is normal.
\par
Now, let us assume that $M$ is normal. Then  \eqref{2.7} is fulfilled. Replacing $x$ and $y$ with $\xi $ in
\eqref{2.7} we have
\begin{equation}\label{401}
F(\xi ,\xi ,z)=0 .
\end{equation}
Replacing $x$  with $\xi $ in \eqref{2.7} and using \eqref{401} we get
\begin{equation}\label{402}
F(\xi ,y,z)=0 .
\end{equation}
Another consecuence from \eqref{2.7} is
\begin{equation}\label{403}
F(\varphi x,\varphi y,\xi )=-F(x,y,\xi ) ,
\end{equation}
which we obtain substituting in \eqref{2.7} $y$ and $z$ with $\varphi y$ and $\xi $, respectively. The equalities \eqref{401} and \eqref{402} mean that the projections $F^{W_4}$ and $F^{W_3}$ of $F$ vanish. Hence, $F(x,y,z)=F^{W_1}(x,y,z)+F^{W_2}(x,y,z)$. Taking into account \eqref{403} we conclude that $F^{W_2}=F^i$  or $F^{W_2}$ is a sum of $F^i$ $(i=5,6,7,8)$. \\
Next, we replace $x$, $y$ and $z$ in \eqref{2.7}  with $\varphi ^2x$, $\varphi ^2y$ and $\varphi z$, respectively. By using \eqref{5} and \eqref{401} we obtain
\begin{equation}\label{404}
F(\varphi ^2x,\varphi ^2y,\varphi ^2z)+F(\varphi x,\varphi y,\varphi ^2z)+F(x,\varphi y,\xi )=0 .
\end{equation}
We substitute $x$ and $y$ in \eqref{404} with $\varphi x$ and $\varphi y$, respectively. So we get
\begin{equation}\label{405}
F(\varphi x,\varphi y,\varphi ^2z)+F(\varphi ^2x,\varphi ^2y,\varphi ^2z)+F(\varphi x,y,\xi )=0 .
\end{equation}
From \eqref{403} by using \eqref{401} we derive $F(x,\varphi y,\xi )=-F(\varphi x,y,\xi )$. Then \eqref{404} and \eqref{405} imply $F(\varphi ^2x,\varphi ^2y,\varphi ^2z)=-F(\varphi x,\varphi y,\varphi ^2z)$, which shows that $F^{W_1}=F^i$  $(i=1,2)$ or $F^{W_1}=F^1+F^2$.  Summarazing the obtained results we conclude that $F=F^i$ $(i=1,2,5,6,7,8)$ or $F=F^1+F^2+F^5+F^6+F^7+F^8$, which completes the proof.
\end{proof}

Now, by using  \eqref{2.3} we compute $d\eta $ for a $(2n+1)$-dimensional almost paracontact metric manifold $M$ belonging to each of the basic classes. We obtain
\begin{lem}\label{Lemma 2.2}
(a) If  $M\in \mathbb{G}_i$, $i=1,2,3,4,6,7,10,11$, then $d\eta =0$ ; \\
(b) If $M\in \mathbb{G}_5$, then $d\eta =\frac{\theta _F(\xi )}{2n}g(\varphi x,y)$ ; \\
(c) If  $M\in \mathbb{G}_i$, $i=8,9$, then $d\eta =-F^i(x,\varphi y,\xi )$ ; \\
(d) If $M\in \mathbb{G}_{12}$, then $d\eta =\frac{1}{2}\left(\eta(x)F(\xi ,\xi ,\varphi y)-
\eta(y)F(\xi ,\xi ,\varphi x)\right)$.
\end{lem}

Let $\overline {\mathbb{G}}_5$ be the subclass of $\mathbb{G}_5$ which consists of all $(2n+1)$-dimensional $\mathbb{G}_5$-manifolds such that $\theta _F(\xi )=2n$
(resp. $\theta _F(\xi )=-2n$) by $\phi (x,y)=g(\varphi x,y)$ (resp. $\phi (x,y)=g(x,\varphi y)$).

Taking into account \lemref{Lemma 2.2} we establish the truth
of the following proposition

\begin{pro}\label{Proposition 2.1}
Let $M$ be a $(2n+1)$-dimensional $\mathbb{G}_5$-manifold. Then $M$ is paracontact metric if and only if $M$ belongs to $\overline {\mathbb{G}}_5$.
\end{pro}
\begin{thm}\label{Theorem 2.3}
A $(2n+1)$-dimensional almost paracontact metric manifold $M(\varphi ,\xi ,\\
\eta ,g)$ is paracontact metric if and only if $M$ belongs to  the class $\overline {\mathbb{G}}_5$  or to the classes which are direct sums of $\overline {\mathbb{G}}_5$ with $\mathbb{G}_4$ and
$\mathbb{G}_{10}$.
\end{thm}
\begin{proof}
Let $M$ belongs to $\overline {\mathbb{G}}_5$. Then from \propref{Proposition 2.1} it follows that $M$ is paracontact metric. If $M$ belongs to a direct sum of $\overline {\mathbb{G}}_5$ with $\mathbb{G}_4$, $\mathbb{G}_{10}$, by using \lemref{Lemma 2.2} we verify that $M$ is also paracontact metric.
\par
Now, let $M$ be a paracontact metric manifold. Then $d\eta (x,y)=\phi (x,y)=g(\varphi x,y)$ and from \eqref{2.3} we have
\begin{equation}\label{406}
-F(x,\varphi y,\xi )+F(y,\varphi x,\xi )=2g(\varphi x,y) .
\end{equation}
Replacing $x$ and $y$ in \eqref{406} with $\xi $ and $\varphi y$, respectively, we get
\begin{equation}\label{407}
F(\xi ,\xi ,y )=0 .
\end{equation}
Moreover, replacing $y$ in \eqref{406} with $\varphi y$ we derive
\begin{equation}\label{408}
F(x,y,\xi )=F(\varphi y,\varphi x,\xi )-2g(\varphi x,\varphi y) .
\end{equation}
From the condition $d\eta =\phi $ it follows that $d\phi =0$. Now, $d\phi (x,y,\xi )=0$ implies
\begin{equation}\label{409}
F(x,y,\xi )-F(y,x,\xi )+F(\xi ,x,y)=0 .
\end{equation}
Substituting \eqref{408} in \eqref{409} we have $F(\varphi y,\varphi x,\xi )-F(\varphi x,\varphi y,\xi )+
F(\xi ,x,y)=0$. In the last equality we replace $x$ with $\varphi x$,  $y$ with $\varphi y$ and by using
\eqref{407} we obtain
\begin{equation}\label{410}
F(y,x,\xi )-F(x,y,\xi )+F(\xi ,x,y)=0 .
\end{equation}
The equalities \eqref{409} and \eqref{410} imply $F(\xi ,x,y)=0$. Also, substituting $F(\xi ,x,y)=0$ in \eqref{409} we get
\begin{equation}\label{411}
F(x,y,\xi )=F(y,x,\xi ) .
\end{equation}
Since $F(\xi ,\xi ,y)=F(\xi ,x,y)=0$, we conclude that $F=F^{W_1}+F^{W_2}$. Hence
\begin{equation}\label{412}
d\eta _F=d\eta _{F^{W_1}}+d\eta _{F^{W_2}} .
\end{equation}
Using \eqref{411} and taking into account the characteristic conditions of the twelve classes we obtain $F^{W_2}=F^5+F^8+F^{10}$. From \lemref{Lemma 2.2} it follows that $d\eta _{F^{W_1}}=0$, $d\eta _{F^5}(x,y)=\frac{\theta  _{F^5}(\xi )}{2n}g(\varphi x,y)$,
$d\eta _{F^8}(x,y)=-F^8(x,\varphi y,\xi )$ and $d\eta _{F^{10}}(x,y)=0$.
Then \eqref{412} becomes
\begin{equation}\label{413}
g(\varphi x,y)=\frac{\theta  _{F^5}(\xi )}{2n}g(\varphi x,y)-F^8(x,\varphi y,\xi ) .
\end{equation}
The equality \eqref{413} implies that either $F^8=0$ or  $F^8=F^5$. The case $F^8=F^5$  leads a contradiction. Therefore
$F^{W_2}=F^5+F^{10}$. Now, because $g$ is non-degenerate,  as an immediate consequence from \eqref{413} we obtain $\theta _{F^5}(\xi )=2n$. This means that $F^5=\overline {F}^5$, where by
$\overline {F}^5$ we denote the projection of $F$ in $\overline {\mathbb{G}}_5$. Thus, for $F^{W_2}$ we obtain $F^{W_2}=\overline {F}^5+F^{10}$ or $F^{W_2}=\overline {F}^5$. The conditions $d\eta _{F^{W_1}}=d\eta _{F^{10}}=0$ and \eqref{412} imply that the case $F^{W_2}=F^{10}$ is impossible. In both cases for $F^{W_2}$ we have $\mathop{\s} \limits_{(x,y,z)}F^{W_2}(x,y,z)=0$.
Then from the equality $\mathop{\s} \limits_{(x,y,z)}F(x,y,z)=\mathop{\s} \limits_{(x,y,z)}F^{W_1}(x,y,z)+\mathop{\s} \limits_{(x,y,z)}F^{W_2}(x,y,z)$ it follows that $\mathop{\s} \limits_{(x,y,z)}F^{W_1}(x,y,z)=0$. Taking into account the characteristic conditions of the classes $\mathbb{G}_i$ $(i=1,2,3,4)$ we conclude that  $F^{W_1}=F^4$. Finally, for $F$ we obtain $F=\overline {F}^5$, or
$F=\overline {F}^5+F^4$, or $F=\overline {F}^5+F^{10}$, or $F=\overline {F}^5+F^4+F^{10}$. Sinse $d\eta _{F^4}=0$, the case $F=F^4$ is impossible.
\end{proof}

As an immediate consequence from \thmref{Theorem 2.2} and \thmref{Theorem 2.3} we obtain
\begin{co}\label{Corollary 2.1}
A $(2n+1)$-dimensional almost paracontact metric manifold $M(\varphi ,\xi ,\\
\eta ,g)$ is para-Sasakian if and only if $M$ belongs to  the class $\overline {\mathbb{G}}_5$.
\end{co}

\begin{rem}\label{Remark 2.1}
The result in  Corollary \ref{Corollary 2.1} is the same as that obtained in \cite{Z}.
\end{rem}

It is known that $\xi $ is Killing vector field if $(L_\xi g)(x,y)=0$. By using  \eqref{2.4} we get
\begin{pro}\label{Proposition 2.2}
The vector field $\xi $ is Killing only in the classes $\mathbb{G}_1$, $\mathbb{G}_2$, $\mathbb{G}_3$,
$\mathbb{G}_4$, $\mathbb{G}_5$, $\mathbb{G}_8$, $\mathbb{G}_9$, $\mathbb{G}_{11}$ and in  the classes which are their direct sums.
\end{pro}
Using \thmref{Theorem 2.3} and \propref{Proposition 2.2} we obtain that among the classes of paracontact metric manifolds the vector field $\xi $ is Killing only in $\overline {\mathbb{G}}_5$ and $\overline {\mathbb{G}}_5\oplus \mathbb{G}_4$. Then we state
\begin{thm}\label{Theorem 2.4}
A $(2n+1)$-dimensional almost paracontact metric manifold $M(\varphi ,\xi ,\\
\eta ,g)$ is K-paracontact metric if and only if $M$ belongs to the classes $\overline {\mathbb{G}}_5$ and  $\overline {\mathbb{G}}_5\oplus \mathbb{G}_4$.
\end{thm}
Finally,  we check that $d\phi$ vanishes only for normal almost paracontact metric manifolds
belonging to the classes $\mathbb{G}_5$, $\mathbb{G}_8$ and $\mathbb{G}_5\oplus \mathbb{G}_8$. Thus we establish the truth of the following theorem
\begin{thm}\label{Theorem 2.5}
A $(2n+1)$-dimensional almost paracontact metric manifold $M(\varphi ,\xi ,\\
\eta ,g)$ is  quasi-para-Sasakian if and only if $M$ belongs to the classes $\mathbb{G}_5$, $\mathbb{G}_8$ and $\mathbb{G}_5\oplus \mathbb{G}_8$.
\end{thm}

\section{The projections of the structure tensor $F$ for dimension 3}\label{sec-4}
Let $(M,\varphi ,\xi ,\eta ,g )$ be a 3-dimensional almost paracontact metric manifold and $\{e_i\}_{i=1}^3=\{e_1,e_2,e_3\}$ be a $\varphi $-basis of
$T_pM$, which satisfies the following conditions:
\[
\begin{array}{ll}
\varphi e_1=e_2, \quad \varphi e_2=e_1, \quad e_3=\xi , \\
g(e_1,e_1)=g(e_3,e_3)=-g(e_2,e_2)=1, \quad g(e_i,e_j)=0, \quad i\neq j \in \{1,2,3\}.
\end{array}
\]
We denote the components of the structure tensor $F$ with respect to the $\varphi $-basis $\{e_i\}_{i=1}^3$ by $F_{ijk}=F(e_i,e_j,e_k)$. By direct computations for arbitrary $x, y, z$, given by $x=x^ie_i$, $y=y^ie_i$, $z=z^ie_i$ with respect to $\{e_i\}_{i=1}^3$, we obtain
\begin{equation}\label {3.1}
\begin{array}{ll}
F(x,y,z)=x^1\left\{F_{113}(y^1z^3-y^3z^1)+F_{123}(y^2z^3-y^3z^2)\right\} \\
+x^2\left\{F_{213}(y^1z^3-y^3z^1)+F_{223}(y^2z^3-y^3z^2)\right\} \\
+x^3\left\{F_{331}(y^3z^1-y^1z^3)+F_{332}(y^3z^2-y^2z^3)\right\} .
\end{array}
\end{equation}
For the components $\theta _F^i=\theta _F(e_i)$, $\theta _F^{* i}=\theta _F^*(e_i)$, $\omega _F^i=\omega (e_i)$ of the Lee forms $\theta _F$, $\theta _F^*$, $\omega _F$ of $F$ we have
\begin{equation}\label {3.2}
\begin{array}{ll}
\theta _F^1=\theta _F^2=\theta _F^{* 1}=\theta _F^{* 2}=0, \quad \theta _F^3=F_{113}-F_{223}, \\
\theta _F^{* 3}=F_{123}-F_{213}, \quad \omega _F^1=F_{331}, \quad \omega _F^2=F_{332}, \quad \omega _F^3=F_{333}=0 .
\end{array}
\end{equation}
\begin{pro}\label {Proposition 3.1}
The structure tensor $F^i$ $(i=1,\ldots ,12)$ of a 3-dimensional almost paracontact metric manifold $(M,\varphi ,\xi ,\eta ,g )$ has the following form in the corresponding basic classes $\mathbb{G}_i$:
\begin{equation}\label {3.3}
\begin{array}{l}
F^1(x,y,z)=F^2(x,y,z)=F^3(x,y,z)=F^4(x,y,z)=0 ; \\ \\
F^5(x,y,z)=\frac{\theta _F^3}{2}\left\{x^1(y^1z^3-y^3z^1)-x^2(y^2z^3-y^3z^2)\right\}, \\
\frac{\theta _F^3}{2}=F_{113}=-F_{131}=-F_{223}=F_{232} ; \\ \\
F^6(x,y,z)=\frac{\theta _F^{* 3}}{2}\left\{x^1(y^2z^3-y^3z^2)-x^2(y^1z^3-y^3z^1)\right\}, \\
\frac{\theta _F^{* 3}}{2}=F_{123}=-F_{132}=-F_{213}=F_{231} ; \\ \\
F^7(x,y,z)=F^8(x,y,z)=F^9(x,y,z)=0 ; \\ \\
F^{10}(x,y,z)=F^{10}_{113}\left\{x^1(y^1z^3-y^3z^1)+x^2(y^2z^3-y^3z^2)\right\} \\
+F^{10}_{123}\left\{x^1(y^2z^3-y^3z^2)+x^2(y^1z^3-y^3z^1)\right\} ,\\
F^{10}_{113}=-F^{10}_{131}=F^{10}_{223}=-F^{10}_{232} , \quad F^{10}_{123}=-F^{10}_{132}=F^{10}_{213}=-F^{10}_{231} ;\\ \\
F^{11}(x,y,z)=0 ; \\ \\
F^{12}(x,y,z)=\omega _F^1x^3(y^3z^1-y^1z^3)+\omega _F^2x^3(y^3z^2-y^2z^3), \\
\omega _F^1=F_{331}=-F_{313}, \quad \omega _F^2=F_{332}=-F_{323} .
\end{array}
\end{equation}
\end{pro}
By using \eqref{3.3} we have
\begin{pro}\label {Proposition 3.2}
The 3-dimensional almost paracontact metric manifolds belong to the classes $\mathbb{G}_5$,
$\mathbb{G}_6$, $\mathbb{G}_{10}$, $\mathbb{G}_{12}$ and to the classes which are their direct sums.
\end{pro}
We note that the assertion  in \propref{Proposition 3.2} follows also from \thmref{Theorem 2.1}.
Taking into account \propref{Proposition 3.2}, \thmref{Theorem 2.2}, \thmref{Theorem 2.3}, \corref{Corollary 2.1},  \propref{Proposition 2.2}, \thmref{Theorem 2.4} and \thmref{Theorem 2.5}
we state:
\begin{thm}\label {Theorem 3.1}
(a) The classes of the 3-dimensional normal almost paracontact metric manifolds are $\mathbb{G}_5$, $\mathbb{G}_6$ and $\mathbb{G}_5\oplus \mathbb{G}_6$; \\
(b) The classes of the 3-dimensional paracontact metric manifolds are $\overline {\mathbb{G}}_5$ and
$\overline {\mathbb{G}}_5\oplus \mathbb{G}_{10}$; \\
(c) The class of the 3-dimensional para-Sasakian manifolds is $\overline {\mathbb{G}}_5$; \\
(d) The class of the 3-dimensional K-paracontact metric manifolds is $\overline {\mathbb{G}}_5$; \\
(e)  The class of the 3-dimensional quasi-para-Sasakian manifolds is $\mathbb{G}_5$.
\end{thm}
\begin{rem}\label{Remark 3.1}
Well known result in the literature is that every $(2n+1)$-dimensional para-Sasakian manifold $M$ is a K-paracontact metric manifold but the converse is true only if $M$ is 3-dimensional. The assertions in
Corollary \ref{Corollary 2.1}, \thmref{Theorem 2.4} and (c), (d) from \thmref{Theorem 3.1} agree with
this result.
\end{rem}

\section{3-dimensional Lie algebras corresponding to Lie groups with almost paracontact metric structure}\label{sec-5}
Let $L$ be a 3-dimensional real connected Lie group and ${\g}$ be its Lie algebra with a basis $\{E_1,E_2,E_3\}$ of left invariant vector fields. We define
an almost paracontact structure $(\varphi ,\xi ,\eta)$ and a semi-Riemannian metric $g$ in the following way:
\[
\begin{array}{llll}
\varphi E_1=E_2 , \quad \varphi E_2=E_1 , \quad \varphi E_3=0 \\
\xi =E_3 , \quad \eta (E_3)=1 , \quad \eta (E_1)=\eta (E_2)=0 , \\
g(E_1,E_1)=g(E_3,E_3)=-g(E_2,E_2)=1 ,\\
\quad g(E_i,E_j)=0, \quad i\neq j \in \{1,2,3\}.
\end{array}
\]
Then $(L,\varphi ,\xi ,\eta ,g)$ is a 3-dimensional almost paracontact metric manifold. Since the metric $g$ is left invariant the Koszul equality becomes
\begin{equation}\label {4.1}
\begin{array}{l}
2g(\nabla _xy,z)=g([x,y],z)+g([z,x],y)+g([z,y],x) ,
\end{array}
\end{equation}
where $\nabla $ is the Levi-Civita connection of $g$. By using \eqref{4.1} we find the components $F_{ijk}=F(E_i,E_j,E_k)$,
$(i,j,k \in \{1,2,3\})$ of the tensor $F$:
\begin{equation}\label {4.2}
\begin{array}{ll}
2F_{ijk}=g([E_i,\varphi E_j]-\varphi [E_i,E_j],E_k)+g([E_k,\varphi E_j]+[\varphi E_k,E_j],E_i) \\
\,\, \, \quad \quad +g([\varphi E_k,E_i]-\varphi[E_k,E_i],E_j) .
\end{array}
\end{equation}
Let the commutators of ${\g}$ be defined by $[E_i,E_j]=C_{ij}^kE_k$, where the structure constants $C_{ij}^k$ are real numbers and $C_{ij}^k=-C_{ji}^k$.
Then from \eqref{4.2} for the non-zero components $F_{ijk}$ we obtain
\begin{equation}\label {4.3}
\begin{array}{llll}
F_{113}=-F_{131}=\frac{1}{2}(C_{12}^3+C_{13}^2-C_{23}^1), \\ \\
F_{223}=-F_{232}=\frac{1}{2}(C_{13}^2-C_{12}^3-C_{23}^1), \\ \\
F_{123}=-F_{132}=-C_{13}^1 , \quad F_{213}=-F_{231}=C_{23}^2 , \\ \\
F_{331}=-F_{313}=C_{23}^3 , \quad F_{332}=-F_{323}=C_{13}^3 .
\end{array}
\end{equation}
Taking into account \eqref{3.2} and \eqref{4.3}, for the non-zero components of $\theta _F$, $\theta _F^*$, $\omega _F$ we have
\begin{equation}\label {4.4}
\theta _F^3=C_{12}^3, \quad \theta _F^{* 3}=-C_{13}^1-C_{23}^2, \quad \omega _F^1=C_{23}^3, \quad \omega _F^2=C_{13}^3.
\end{equation}
Using \eqref{3.3}, \eqref{4.3}, \eqref{4.4} and applying the Jacobi identity
\[
\mathop{\s} \limits_{E_i,E_j,E_k}\bigl[[E_i,E_j],E_k\bigr]=0
\]
we deduce the following
\begin{thm}\label{Theorem 4.1}
The manifold $(L,\varphi ,\xi ,\eta ,g)$ belongs to the class $\mathbb{G}_i (i\in \{5,6,10,12\})$ if and only if the corresponding Lie algebra ${\g}$ is determined
by the following commutators:
\begin{equation}\label {4.5}
\begin{array}{ll}
{\bf\mathbb{G}_5} : [E_1,E_2]=C_{12}^1E_1+C_{12}^2E_2+C_{12}^3E_3, \quad [E_1,E_3]=C_{13}^2E_2, \\ \\

[E_2,E_3]=C_{13}^2E_1 : \quad \theta _{F_5}^3=C_{12}^3\neq 0, \quad C_{12}^1C_{13}^2=0, \quad C_{12}^2C_{13}^2=0;
\end{array}
\end{equation}
\begin{equation}\label {4.6}
\begin{array}{ll}
{\bf\mathbb{G}_6} : [E_1,E_2]=C_{12}^1E_1+C_{12}^2E_2, \quad [E_1,E_3]=C_{13}^1E_1+C_{13}^2E_2, \\ \\

[E_2,E_3]=C_{13}^2E_1+C_{13}^1E_2 : \theta _{F_6}^{* 3}=-2C_{13}^1\neq 0, \\ \\
C_{12}^2C_{13}^2-C_{13}^1C_{12}^1=0, \quad C_{12}^1C_{13}^2-C_{13}^1C_{12}^2=0;
\end{array}
\end{equation}
\begin{equation}\label {4.7}
\begin{array}{ll}
{\bf\mathbb{G}_{10}} : [E_1,E_2]=C_{12}^1E_1+C_{12}^2E_2, \quad [E_1,E_3]=C_{13}^1E_1+C_{13}^2E_2, \\ \\

[E_2,E_3]=C_{23}^1E_1-C_{13}^1E_2 : C_{13}^2\neq C_{23}^1 \quad {\text {\rm or}} \quad C_{13}^1\neq 0, \\ \\
C_{12}^1C_{13}^1+C_{12}^2C_{23}^1=0, \quad C_{12}^1C_{13}^2-C_{12}^2C_{13}^1=0;
\end{array}
\end{equation}
\begin{equation}\label {4.8}
\begin{array}{ll}
{\bf\mathbb{G}_{12}} : [E_1,E_2]=C_{12}^1E_1+C_{12}^2E_2, \quad [E_1,E_3]=C_{13}^2E_2+C_{13}^3E_3, \\ \\

[E_2,E_3]=C_{13}^2E_1+C_{23}^3E_3 : C_{13}^3\neq 0 \quad {\text {\rm or}} \quad C_{23}^3\neq 0, \\ \\

(C_{12}^1-C_{23}^3)C_{13}^2=0, \quad (C_{12}^2+C_{13}^3)C_{13}^2=0, \\ \\

(C_{12}^1-C_{23}^3)C_{13}^3+(C_{12}^2+C_{13}^3)C_{23}^3=0.
\end{array}
\end{equation}
\end{thm}

\section{Matrix Lie groups as 3-dimensional almost paracontact metric manifolds}\label{sec-6}
Let $(L,\varphi ,\xi ,\eta ,g)$ be a 3-dimensional almost paracontact metric manifold from \secref{sec-5} belonging to some of the classes $\mathbb{G}_i (i\in \{5,6,10,12\})$.
By $G$  we denote the simply connected Lie group isomorphic to $L$, both with one and the same Lie algebra ${\g}$. Further, we find the
adjoint representation $\rm Ad$ of $G$, which is the following Lie group homomorphism
\[
\rm {Ad} : G \longrightarrow Aut({\g}).
\]
For $X\in {\g}$, the map ${\rm {ad}}_X : {\g}\longrightarrow {\g}$ is defined by ${\rm {ad}}_X(Y)=[X,Y]$, where by ${\rm ad}_X$ is denoted as ${\rm {ad}}(X)$. Due to the Jacobi identity, the map
\[
\rm {ad} : {\g} \longrightarrow End({\g}) : X\longrightarrow ad_X
\]
is Lie algebra homomorphism, which is called  adjoint representation of ${\g}$.
Since the set ${\rm End}({\g})$ of all ${\K}$-linear maps from ${\g}$ to ${\g}$ is isomorphic to the set of all  $(n\times n)$ matrices ${\rm M}(n,{\K})$ with entries in ${\K}$, $\rm {ad}$ is a matrix representation of ${\g}$. We denote by $M_i$ the matrices of ${\rm ad}_{E_i}$ (i=1,2,3) with respect to the basis $\{E_1,E_2,E_3\}$ of ${\g}$.  Then for an arbitrary $X=aE_1+bE_2+cE_3$ ($a, b, c \in {\R}$) in ${\g}$ the matrix $A$ of ${\rm ad}_X$ is $A=aM_1+bM_2+cM_3$.
By using the well known identity $e^A={\rm {Ad}}\left(e^X\right)$ we find the matrix representation of the Lie group $G$.
\subsection{Matrix Lie groups as manifolds from the class $\mathbb{G}_5$}\label{subsec-6.1}
Let ${\g}_5$ be the Lie algebra obtained from \eqref{4.5} by $C_{12}^1=C_{12}^2=C_{13}^2=0$, i.e.
\begin{equation}\label{5.1.1}
[E_1,E_2]=\alpha E_3, \quad [E_1,E_3]=0, \quad [E_2,E_3]=0,
\end{equation}
where $\alpha =\theta _{F_5}^3=C_{12}^3$.
Then from \thmref{Theorem 4.1} it follows that $(L,\varphi ,\xi ,\eta ,g)$, where $L$ is a Lie group with a Lie algebra  ${\g}_5$, belongs to the class $\mathbb{G}_5$.
In partucular, the manifold  is paracontact metric if and only if $\alpha =2$.
The Levi-Civita connection ${\nabla}$ is given by
\[
\begin{array}{llll}
\nabla_{E_1}E_1=0 , \quad \nabla_{E_1}E_2=\frac{\alpha}{2}E_3 , \quad \nabla_{E_1}E_3=\frac{\alpha}{2}E_2 , \\
\nabla_{E_2}E_1=-\frac{\alpha}{2}E_3 , \quad \nabla_{E_2}E_2=0 , \quad \nabla_{E_2}E_3=\frac{\alpha}{2}E_1  , \\
\nabla_{E_3}E_1=\frac{\alpha}{2}E_2 , \quad \nabla_{E_3}E_2=\frac{\alpha}{2}E_1 , \quad \nabla_{E_3}E_3=0.   \\
\end{array}
\]
It is not hard to see that the Ricci tensor $Ric$ is equal to
$$Ric(x,y)=scalg(x,y)-2scal\eta(x)\eta(y),$$
where $scal=\frac{\alpha^2}{2}$ is the scalar curvature. Consequently, $L$ is an $\eta$-Einstein manifold.
For the matrices $M_i$ \, (i=1,2,3) and $A$  we have:
\[
M_1=\left(\begin{array}{lll}
0 & 0 & 0 \cr
0 & 0  & 0 \cr
0 & \alpha & 0
\end{array}\right) , \quad
M_2=\left(\begin{array}{cll}
0 & 0 & 0 \cr
0 & 0  & 0 \cr
-\alpha  & 0 & 0
\end{array}\right) , \quad
M_3=\left(\begin{array}{lll}
0 & 0 & 0 \cr
0 & 0 & 0 \cr
0 & 0 & 0
\end{array}\right) ,
\]
\begin{equation}\label{5.1.2}
A=\left(\begin{array}{ccl}
0 & 0 & 0 \cr
0 & 0 & 0 \cr
-b\alpha & a\alpha & 0
\end{array}\right) .
\end{equation}
\begin{thm}\label{Theorem 5.1.1}
The matrix representation of the Lie group $G_5$ corresponding to the Lie algebra ${\g}_5$, determined by \eqref{5.1.1} and having the matrix representation
\eqref{5.1.2}, is
\begin{equation}\label{5.1.3}
G_5=\left\{e^A=
\left(\begin{array}{ccl}
1 & 0 & 0 \cr
0 & 1 & 0 \cr
-b\alpha  & a\alpha & 1
\end{array}\right)
\right\} .
\end{equation}
\end{thm}
\begin{proof}
The matrix $A$ is nilpotent of degree $q=2$. Therefore we compute $e^A$ directly from
\begin{equation}\label {5.1}
\begin{array}{l}
e^A=E+A+\frac{A^2}{2!}+\frac{A^3}{3!}+\ldots +\frac{A^{q-1}}{(q-1)!} \, ,
\end{array}
\end{equation}
i.e. $e^A=E+A$. So we obtain \eqref{5.1.3}.
\end{proof}

\subsection{Matrix Lie groups as manifolds from the class $\mathbb{G}_6$}\label{subsec-6.2}
We consider the  Lie algebra ${\g}_6$ obtained from \eqref{4.6} by
$C_{12}^1=C_{12}^2=0$, i.e.
\begin{equation}\label{5.2.1}
[E_1,E_2]=0, \quad [E_1,E_3]=\alpha E_1+\beta E_2, \quad [E_2,E_3]=\beta E_1+\alpha E_2,
\end{equation}
where $\alpha =-\frac{\theta _{F_6}^{*3}}{2}=C_{13}^1, \, \, \beta =C_{13}^2\neq 0$.
Then from \thmref{Theorem 4.1} it follows that $(L,\varphi ,\xi ,\eta ,g)$, where $L$ is a Lie group with a Lie algebra  ${\g}_6$, belongs to the class $\mathbb{G}_6$.
The Levi-Civita connection ${\nabla}$ is given by
\[
\begin{array}{llll}
\nabla_{E_1}E_1=-\alpha E_3 , \quad \nabla_{E_1}E_2=0 , \quad \nabla_{E_1}E_3=\alpha E_1 , \\
\nabla_{E_2}E_1=0 , \quad \nabla_{E_2}E_2=\alpha E_3 , \quad \nabla_{E_2}E_3=\alpha E_2  , \\
\nabla_{E_3}E_1=-\beta E_2 , \quad \nabla_{E_3}E_2=-\beta E_1 , \quad \nabla_{E_3}E_3=0.   \\
\end{array}
\]
It not hard to see that the Ricci tensor $Ric$ is equal to
$$Ric(x,y)=\frac{scal}{3}g(x,y),$$
where $scal=-6\alpha^2$ is the scalar curvature. Consequently, $L$ is an Einstein manifold.

For the matrices $M_i$ \, (i=1,2,3) and $A$ we get:
\[
M_1=\left(\begin{array}{llc}
0 & 0 & \alpha \cr
0 & 0  & \beta \cr
0 & 0 & 0
\end{array}\right) , \quad
M_2=\left(\begin{array}{llr}
0 & 0 & \beta \cr
0 & 0 & \alpha \cr
0 & 0 & 0
\end{array}\right) , \quad
M_3=\left(\begin{array}{rrl}
-\alpha & -\beta & 0 \cr
-\beta & -\alpha & 0 \cr
0 & 0 & 0
\end{array}\right) ,
\]
\begin{equation}\label{5.2.2}
A=\left(\begin{array}{rrc}
-c\alpha & -c\beta & a\alpha +b\beta \cr
-c\beta & -c\alpha & b\alpha +a\beta\cr
0 & 0 & 0
\end{array}\right) .
\end{equation}
\begin{thm}\label{Theorem 5.2.1}
The matrix representation of the Lie group $G_6$ corresponding to the Lie algebra ${\g}_6$, determined by \eqref{5.2.1} and having the matrix representation
\eqref{5.2.2}, is as follows:
\begin{itemize}
\item If $c\neq 0$, $\beta =\alpha $ and $b=-a$, then
\begin{equation}\label{5.2.3}
G_6=\left\{e^A=
\left(\begin{array}{ccc}
\frac{1+e^{-2c\alpha }}{2} & \frac{-1+e^{-2c\alpha }}{2} & 0 \cr \cr
\frac{-1+e^{-2c\alpha }}{2} & \frac{1+e^{-2c\alpha }}{2} & 0 \cr \cr
\frac{-1+e^{-2c\alpha }}{2} & \frac{-1+e^{-2c\alpha }}{2} & 1
\end{array}\right)
\right\} .
\end{equation}
\item If $c\neq 0$, $\beta =\alpha $ and $b\neq-a$, then
\begin{equation}\label{5.2.4}
G_6=\left\{e^A=
\left(\begin{array}{ccc}
\frac{1+e^{-2c\alpha }}{2} & \frac{-1+e^{-2c\alpha }}{2} & \frac{(a+b)(1-e^{-2c\alpha })}{2c} \cr \cr
\frac{-1+e^{-2c\alpha }}{2} & \frac{1+e^{-2c\alpha }}{2} &  \frac{(a+b)(1-e^{-2c\alpha })}{2c}\cr \cr
0 & 0 & 1
\end{array}\right)
\right\} .
\end{equation}
\item If $c\neq 0$, $\beta =-\alpha $ and $b=a$, then
\begin{equation}\label{5.2.5}
G_6=\left\{e^A=
\left(\begin{array}{ccc}
\frac{1+e^{-2c\alpha }}{2} & \frac{1-e^{-2c\alpha }}{2} & 0 \cr \cr
\frac{1-e^{-2c\alpha }}{2} & \frac{1+e^{-2c\alpha }}{2} & 0 \cr \cr
\frac{-1+e^{-2c\alpha }}{2} & \frac{1-e^{-2c\alpha }}{2} & 1
\end{array}\right)
\right\} .
\end{equation}
\item If $c\neq 0$, $\beta =-\alpha $ and $b\neq a$, then
\begin{equation}\label{5.2.6}
G_6=\left\{e^A=
\left(\begin{array}{ccc}
\frac{1+e^{-2c\alpha }}{2} & \frac{1-e^{-2c\alpha }}{2} & \frac{(a-b)(1-e^{-2c\alpha })}{2c} \cr \cr
\frac{1-e^{-2c\alpha }}{2} & \frac{1+e^{-2c\alpha }}{2} &  \frac{(a-b)(-1+e^{-2c\alpha })}{2c}\cr \cr
0 & 0 & 1
\end{array}\right)
\right\} .
\end{equation}
\item If $c\neq 0$, $\beta \neq \pm \alpha $, then
\begin{equation}\label{5.2.7}
\small{G_6=
\left(\begin{array}{ccc}
e^{-c\alpha }\cosh c\beta & -e^{-c\alpha }\sinh c\beta  & \frac{a(1-e^{-c\alpha }\cosh c\beta)+be^{-c\alpha }\sinh c\beta }{c} \cr \cr
-e^{-c\alpha }\sinh c\beta  & e^{-c\alpha }\cosh c\beta & \frac{b(1-e^{-c\alpha }\cosh c\beta)+ae^{-c\alpha }\sinh c\beta }{c} \cr \cr
0 & 0 & 1
\end{array}\right) .}
\end{equation}
\item If $c=0$, then
\begin{equation}\label{5.2.8}
G_6=\left\{e^A=
\left(\begin{array}{llc}
1 & 0 & a\alpha +b\beta \cr
0 & 1  & b\alpha +a\beta \cr
0 & 0 & 1
\end{array}\right)
\right\} .
\end{equation}
\end{itemize}
\end{thm}
\begin{proof}
The characteristic polynomial of A is
\[
P_A(\lambda )=\lambda(-c\alpha +c\beta -\lambda )(c\alpha +c\beta+\lambda ) =0 .
\]
Hence for the eigenvalues $\lambda _i \, (i = 1, 2, 3)$ of $A$ we have
\[
\lambda _1=0 , \quad \lambda _2=c(\beta -\alpha ) , \quad \lambda _3=-c(\alpha +\beta ) .
\]
First, we assume that $c\neq 0$. If $\beta =\alpha $ \, (resp. $\beta =-\alpha $) we obtain
\[
\lambda _1=\lambda _2=0, \, \lambda _3=-2c\alpha  \qquad
(\text {resp.} \, \lambda _1=\lambda _3=0, \, \lambda _2=-2c\alpha ).
\]
Let us consider the case $\beta =\alpha $. Then the eigenvectors
\[
p_1=(-1,1,0), \quad p_2=(a+b,0,c),
\]
corresponding to $\lambda _1=\lambda _2=0$, are linearly independent for arbitrary $a$ and $b$. The coordinates $(x_1,x_2,x_3)$ of the eigenvector $p_3$, corresponding to $\lambda _3=-2c\alpha $,
satisfy the following system:
\[
\left|\begin{array}{ll}
x_1-x_2+\frac{a+b}{c}x_3=0 \cr \cr
-x_1+x_2+\frac{a+b}{c}x_3=0 .
\end{array}\right.
\]
If we suppose that $b=-a$, then $p_1, p_2=(0,0,c)$ and $p_3=(1,1,1)$ are linearly independent and for the change of basis matrix P we get
\[
P=\left(\begin{array}{rll}
-1 & 0 & 1 \cr
1 & 0 & 1 \cr
0 & c & 1
\end{array}\right) .
\]
By using that $e^A = Pe^JP^{-1}$, where $J$ is the diagonal matrix with elements $J_{ii} =\lambda _i$ and $P^{-1}$ is the inverse matrix of $P$,
we obtain the matrix representation \eqref{5.2.3} of $G_6$. When $b\neq -a$ the vectors  $p_1, p_2=(a+b,0,c)$ and $p_3=(1,1,0)$ are linearly independent
and the matrix representation of $G_6$ is in the form \eqref{5.2.4}.\\
In the case $\beta =-\alpha$, by analogical computations, we obtain \eqref{5.2.5} and \eqref{5.2.6}.\\
Now, if we take $\beta \neq \alpha$, then the eigenvalues $\lambda _i \, (i = 1, 2, 3)$ of $A$ are different and hence the corresponding eigenvectors
\[
p_1=\left(\frac{a}{c},\frac{b}{c},1\right), \quad p_2=(1,-1,0) \quad \text{and} \quad p_3=(1,1,0)
\]
are linearly independent. Then the matrix representation of $G_6$ is \eqref{5.2.7}.\\
Finally, the assumption $c=0$ implies that $A$ is nilpotent matrix of degree \\ q = 2  and by using \eqref{5.1} we obtain \eqref{5.2.8}.
\end{proof}

\subsection{Matrix Lie groups as manifolds from the class $\mathbb{G}_{10}$}\label{subsec-6.3}
We consider the  Lie algebra ${\g}_{10}$ obtained from \eqref{4.7} by
$C_{12}^1=C_{12}^2=C_{13}^2=C_{23}^1=0$, i.e.
\begin{equation}\label{5.3.1}
[E_1,E_2]=0, \quad [E_1,E_3]=\alpha E_1, \quad [E_2,E_3]=-\alpha E_2,
\end{equation}
where $\alpha =C_{13}^1\neq 0$.
Then from \thmref{Theorem 4.1} it follows that $(L,\varphi ,\xi ,\eta ,g)$, where $L$ is a Lie group with a Lie algebra  ${\g}_{10}$, belongs to the class $\mathbb{G}_{10}$.
The Levi-Civita connection ${\nabla}$ is given by
\[
\begin{array}{llll}
\nabla_{E_1}E_1=-\alpha E_3 , \quad \nabla_{E_1}E_2=0 , \quad \nabla_{E_1}E_3=\alpha E_1 , \\
\nabla_{E_2}E_1=0 , \quad \nabla_{E_2}E_2=-\alpha E_3 , \quad \nabla_{E_2}E_3=-\alpha E_2  , \\
\nabla_{E_3}E_1=0 , \quad \nabla_{E_3}E_2=0 , \quad \nabla_{E_3}E_3=0.   \\
\end{array}
\]
It is easy to see that the Ricci tensor $Ric$ is equal to
$$Ric(x,y)=scal\eta(x)\eta(y),$$
where $scal=-6\alpha^2$ is the scalar curvature.

For the matrices $M_i$ \, (i=1,2,3) and $A$ we have:
\[
M_1=\left(\begin{array}{llc}
0 & 0 & \alpha \cr
0 & 0  & 0 \cr
0 & 0 & 0
\end{array}\right) , \quad
M_2=\left(\begin{array}{llr}
0 & 0 & 0 \cr
0 & 0 & -\alpha \cr
0 & 0 & 0
\end{array}\right) , \quad
M_3=\left(\begin{array}{rll}
-\alpha & 0 & 0 \cr
0 & \alpha & 0 \cr
0 & 0 & 0
\end{array}\right) ,
\]
\begin{equation}\label{5.3.2}
A=\left(\begin{array}{rcr}
-c\alpha & 0 & a\alpha \cr
0 & c\alpha & -b\alpha \cr
0 & 0 & 0
\end{array}\right) .
\end{equation}
\begin{thm}\label{Theorem 5.3.1}
The matrix representation of the Lie group $G_{10}$ corresponding to the Lie algebra ${\g}_{10}$, determined by \eqref{5.3.1} and having the matrix representation
\eqref{5.3.2}, is as follows:
\begin{itemize}
\item If $c\neq 0$, then
\begin{equation}\label{5.3.3}
G_{10}=\left\{e^A=
\left(\begin{array}{ccc}
e^{-c\alpha } & 0 & \frac{a\left(1-e^{-c\alpha }\right)}{c} \cr \cr
0 & e^{c\alpha } & \frac{b\left(1-e^{c\alpha }\right)}{c} \cr \cr
0 & 0 & 1
\end{array}\right)
\right\} .
\end{equation}
\item If $c=0$, then
\begin{equation}\label{5.3.4}
G_{10}=\left\{e^A=\left(\begin{array}{llc}
1 & 0 & a\alpha \cr
0 & 1 & -b\alpha \cr
0 & 0 & 1
\end{array}\right)\right\} .
\end{equation}
\end{itemize}
\end{thm}
\begin{proof}
From the characteristic polynomial of A
\[
P_A(\lambda )=(-c\alpha -\lambda )(c\alpha-\lambda )\lambda =0
\]
we find
\[
\lambda _1=-c\alpha , \quad \lambda _2=c\alpha , \quad \lambda _3=0 .
\]
If $c\neq 0$, then the eigenvalues $\lambda _i \, (i = 1, 2, 3)$ of $A$ are different and hence the corresponding to $\lambda _i \, (i = 1, 2, 3)$ eigenvectors
\[
p_1=(1,0,0), \quad p_2=(0,1,0), \quad p_3=(a,b,c)
\]
are linearly independent. For the change of basis matrix P we have
\[
P=\left(\begin{array}{lll}
1 & 0 & a \cr
0 & 1 & b \cr
0 & 0 & c
\end{array}\right) .
\]
Then for the matrix representation of $G_{10}$ we obtain \eqref{5.3.3}.\\
In the case when $c=0$ the matrix $A$ is nilpotent of degree $q=2$. Using \eqref{5.1} we establish \eqref{5.3.4}.
\end{proof}

\subsection{Matrix Lie groups as manifolds from the class $\mathbb{G}_{12}$}\label{subsec-6.4}
Let ${\g}_{12}$ be the  Lie algebra  obtained from \eqref{4.8} by $C_{13}^2=0$, i.e.
\begin{equation}\label{5.4.1}
[E_1,E_2]=\alpha E_1+\beta E_2, \quad [E_1,E_3]=-\beta E_3, \quad [E_2,E_3]=\alpha E_3,
\end{equation}
where $\alpha =\omega _F^1=C_{23}^3=C_{12}^1\neq 0$, $\beta =-\omega _F^2=-C_{13}^3=C_{12}^2\neq 0$.
Then from \thmref{Theorem 4.1} it follows that $(L,\varphi ,\xi ,\eta ,g)$, where $L$ is a Lie group with a Lie algebra  ${\g}_{12}$, belongs to the class $\mathbb{G}_{12}$.
The Levi-Civita connection ${\nabla}$ is given by
\[
\begin{array}{llll}
\nabla_{E_1}E_1=\alpha E_2 , \quad \nabla_{E_1}E_2=\alpha E_1 , \quad \nabla_{E_1}E_3=0 , \\
\nabla_{E_2}E_1=-\beta E_2 , \quad \nabla_{E_2}E_2=-\beta E_1 , \quad \nabla_{E_2}E_3=0  , \\
\nabla_{E_3}E_1=\beta E_3 , \quad \nabla_{E_3}E_2=-\alpha E_3 , \quad \nabla_{E_3}E_3=-\beta E_1-\alpha E_2.   \\
\end{array}
\]

Hence the matrices $M_i$ \, (i=1,2,3) and $A$ are:
\[
M_1=\left(\begin{array}{llr}
0 & \alpha & 0 \cr
0 & \beta   & 0\cr
0 & 0 & -\beta
\end{array}\right) , \quad
M_2=\left(\begin{array}{rlr}
-\alpha & 0 & 0 \cr
-\beta & 0 & 0 \cr
0 & 0 & \alpha
\end{array}\right) , \quad
M_3=\left(\begin{array}{rrl}
0 & 0 & 0 \cr
0 & 0 & 0 \cr
\beta & -\alpha & 0
\end{array}\right) ,
\]
\begin{equation}\label{5.4.2}
A=\left(\begin{array}{rrc}
-b\alpha & a\alpha & 0 \cr
-b\beta & a\beta & 0\cr
c\beta & -c\alpha & b\alpha-a\beta
\end{array}\right) .
\end{equation}
\begin{thm}\label{Theorem 5.4.1}
The matrix representation of the Lie group $G_{12}$ corresponding to the Lie algebra ${\g}_{12}$, determined by \eqref{5.4.1} and having the matrix representation
\eqref{5.4.2}, is as follows:
\begin{itemize}
\item If $b\alpha-a\beta\neq 0$, then
\begin{equation}\label{5.4.3}
G_{12}=\left\{e^A=
\left(\begin{array}{llc}
\frac{a\beta -b\alpha e^{a\beta -b\alpha }}{a\beta -b\alpha} & \frac{a\alpha (e^{a\beta -b\alpha}-1)}{a\beta -b\alpha} & 0 \cr \cr
\frac{b\beta (1-e^{a\beta -b\alpha})}{a\beta -b\alpha} & \frac{-b\alpha +a\beta e^{a\beta -b\alpha }}{a\beta -b\alpha}  & 0 \cr \cr
\frac{c\beta (1-e^{b\alpha -a\beta })}{a\beta -b\alpha} & \frac{c\alpha (e^{b\alpha -a\beta }-1)}{a\beta -b\alpha} & e^{b\alpha -a\beta }
\end{array}\right)
\right\} ,
\end{equation}
where $(a,b)\neq (0,0)$.
\item If $b\alpha-a\beta =0$, then
\begin{equation}\label{5.4.4}
G_{12}=\left\{e^A=
\left(\begin{array}{ccc}
1-b\alpha  & a\alpha  & 0 \cr
-b\beta  & 1+a\beta   & 0 \cr
c\beta  & -c\alpha & 1
\end{array}\right)
\right\} ,
\end{equation}
where both $a$ and $b$ are zero or non-zero.
\end{itemize}
\end{thm}
\begin{proof}
From the characteristic polynomial of A
\[
P_A(\lambda )=\lambda (\lambda +b\alpha -a\beta )(b\alpha -a\beta -\lambda ) =0
\]
we find
\[
\lambda _1=0 , \quad \lambda _2=-b\alpha +a\beta , \quad \lambda _3=b\alpha -a\beta .
\]
First, we assume that $b\alpha-a\beta\neq 0$. From this condition it follows that $(a,b)\neq (0,0)$ and  the eigenvalues $\lambda _i \, (i = 1, 2, 3)$ of $A$ are different. Then the corresponding to $\lambda _i \, (i = 1, 2, 3)$ eigenvectors
\[
p_1=(a,b,c), \quad p_2=(\alpha ,\beta ,0), \quad p_3=(0,0,1)
\]
are linearly independent and the change of basis matrix P is
\[
P=\left(\begin{array}{lll}
a & \alpha & 0 \cr
b & \beta & 0 \cr
c & 0 & 1
\end{array}\right) .
\]
By straightforward computations we obtain that in this case \eqref{5.4.3} is the matrix representation of $G_{12}$.\\
If $b\alpha-a\beta =0$, then  both $a$ and $b$ are zero or non-zero. In this case the matrix $A$ is nilpotent of degree $q=2$ and  the matrix representation of
$G_{12}$ is in the form \eqref{5.4.4}.
\end{proof}

\section*{Acknowledgments}

S.Z. is partially supported by Contract DFNI I02/4/12.12.2014 and Contract 80-10-33/2017 with the Sofia University ''St.Kl.Ohridski''.\\

G.N. is partially supported by Contract FSD-31-653-08/19.06.2017 with the University of Veliko Tarnovo "St. Cyril and St. Methodius".

\end{document}